\documentclass[12pt]{amsart}

\usepackage{amsfonts, amsmath, amscd}

\usepackage[psamsfonts]{amssymb}

\usepackage{amssymb}

\usepackage{pb-diagram}

\usepackage[usenames]{color}

\newtheorem{theorem}{Theorem}[section]
\newtheorem{lemma}[theorem]{Lemma}
\newtheorem{corollary}[theorem]{Corollary}
\newtheorem{question}[theorem]{Question}
\newtheorem{remark}[theorem]{Remark}
\newtheorem{proposition}[theorem]{Proposition}

\newtheorem{example}[theorem]{Example}
\newtheorem{fact}[theorem]{Fact}

\numberwithin{equation}{section}

\newcommand{\CC}{C_k}
\newcommand{\NN}{\mathbb{N}}

\newcommand{\GG}{\mathfrak{G}}

\newcommand{\KK}{\mathcal{K}}

\newcommand{\IR}{\mathbb{R}}
\newcommand{\II}{\mathbb{I}}

\newcommand{\eps}{\varepsilon}
\renewcommand{\phi}{\varphi}

\newcommand{\U}{\mathcal U}

\newcommand{\dx}{\;{\rm d}}
\newcommand{\er}{\mathbb R}

\newcommand{\en}{\mathbb N}
\newcommand{\sub}{\subseteq}

\input xy
\xyoption{all}

\title[The Ascoli property for function spaces]{The Ascoli property for function spaces and \\ the weak topology of Banach and Fr\'echet spaces}

\author{S.~Gabriyelyan}
\address{Department of Mathematics, Ben-Gurion University of the
Negev, Beer-Sheva, P.O. 653, Israel}
\email{saak@math.bgu.ac.il}

\author{J. K{\c{a}}kol}
\thanks{The second named author was supported by the Center for Advanced Studies in Mathematics at Ben-Gurion University of the Negev
and by Generalitat Valenciana, Conselleria d'Educaci\'{o}, Cultura i Esport, Spain, Grant PROMETEO/2013/058.}
\address{Faculty of Mathematics and Informatics A. Mickiewicz University, $61-614$ Pozna{\'n}, Poland}
\email{kakol@amu.edu.pl}

\author{G. Plebanek}
\address{Instytut Matematyczny, Uniwersytet Wroc{\l}awski, Wroc{\l}aw, Poland}
\email{grzes@math.uni.wroc.pl}
 \thanks{The third name author was partially supported by NCN grant 2013/11/B/ST1/03596 (2014-2017).}

\begin{document}

\begin{abstract}
Following \cite{BG} we say that a Tychonoff space $X$  is an Ascoli space  if every  compact subset $\KK$ of $\CC(X)$ is evenly continuous; this notion is closely related to the classical Ascoli theorem.  Every $k_\IR$-space, hence any $k$-space,  is Ascoli. 

Let $X$ be a metrizable space. We prove that the space $C_{k}(X)$ is Ascoli iff $C_{k}(X)$ is a $k_\IR$-space iff $X$ is locally compact. Moreover,
$C_{k}(X)$ endowed with the weak topology is Ascoli iff $X$ is countable and discrete.

Using some basic concepts from probability theory and measure-theoretic properties of $\ell_1$, we show that the following assertions are equivalent for a Banach space $E$: (i)  $E$ does not contain isomorphic copy of $\ell_1$, (ii) every real-valued sequentially continuous map on the unit ball $B_{w}$ with the weak topology is continuous, (iii) $B_{w}$ is a  $k_\IR$-space, (iv) $B_{w}$ is an Ascoli space.


We prove also that a Fr\'{e}chet lcs  $F$ does not contain isomorphic copy of $\ell_1$ iff each closed and convex bounded subset of $F$ is Ascoli in the weak topology. However we show that a Banach space $E$ in the weak topology is Ascoli iff $E$ is finite-dimensional. We supplement the last result by showing that a Fr\'{e}chet lcs $F$ which is a quojection is Ascoli in the weak topology iff either $F$ is finite dimensional or $F$ is isomorphic to the product $\mathbb{K}^{\mathbb{N}}$, where $\mathbb{K}\in\{\mathbb{R},\mathbb{C}\}$.
\end{abstract}

\maketitle

\section{Introduction}

Several  topological properties of function spaces have been intensively studied for many years, see for instance  \cite{Arhangel,kak,mcoy}  and references therein. In particular, various topological properties generalizing metrizability attracted a lot of attention. Let us  mention,  for example,   Fr\'{e}chet--Urysohn property, sequentiality, $k$-space property and $k_\IR$-space property  (all relevant definitions are given in Section \ref{sec:2} below).
 It is well known that
\[
\xymatrix{
\mbox{metric} \ar@{=>}[r] & {\mbox{Fr\'{e}chet--}\atop\mbox{Urysohn}} \ar@{=>}[r] & \mbox{sequential} \ar@{=>}[r] &  \mbox{$k$-space} \ar@{=>}[r] &  \mbox{$k_\IR$-space}} ,
\]
and none of these implications is reversible (see \cite{Eng,Mi73}).

For topological spaces $X$ and $Y$, we denote by $\CC(X,Y)$ the space $C(X,Y)$ of all continuous  functions from $X$ into $Y$ endowed with the compact-open topology. For $\II =[0,1]$, Pol \cite{Pol-1974} proved the following remarkable result
\begin{theorem}[\cite{Pol-1974}] \label{t:Pol-k-space}
Let $X$ be a first countable paracompact space. Then the space $\CC(X,\II)$ is a $k$-space  if and  only if $X=L\cup D$ is the topological sum of a locally compact Lindel\"{o}f space $L$ and a discrete space $D$.
\end{theorem}
Theorem \ref{t:Pol-k-space} easily implies the following result noticed in \cite{Gabr-C2}.
\begin{corollary} \label{c:Pol-k-space}
For a metric space $X$, the space $\CC(X)$ is a $k$-space  if and  only if $\CC(X)$ is a Polish space if and only if $X$ is a Polish locally compact space.
\end{corollary}
Note also that by a result of Pytkeev \cite{Pyt2}, for a topological space $X$ the space $\CC(X)$ is a $k$-space if and only if it is Fr\'{e}chet--Urysohn. For a metrizable space $X$ and the doubleton $\mathbf{2}=\{ 0,1\}$, topological properties of the space $\CC(X,\mathbf{2})$  are thoroughly studied in \cite{Gabr-C2}.

For a topological space $X$, denote by $\psi: X\times\CC(X)\rightarrow \IR$, $\psi(x,f):=f(x)$, the evaluation map. Recall that a subset $\KK$ of $\CC(X)$ is {\em evenly continuous} if the restriction of $\psi$ onto $X\times \KK$ is jointly continuous, i.e. for any $x\in X$, each $f\in\KK$  and every neighborhood $O_{f(x)}\subset Y$ of $f(x)$ there exist neighborhoods $U_f\subset \KK$ of $f$ and $O_x\subset X$ of $x$ such that $U_f(O_x):=\{g(y):g\in U_f,\;y\in O_x\}\subset O_{f(x)}$.

Following \cite{BG}, a  Tychonoff (Hausdorff) space $X$ is called an {\em Ascoli space} if each compact subset $\KK$ of $\CC(X)$ is evenly continuous. In other words, $X$ is Ascoli if and only if the compact-open topology of $\CC(X)$ is Ascoli in the sense of \cite[p.45]{mcoy}.

It is easy to see that a space $X$ is Ascoli if and only if the canonical valuation map $X\hookrightarrow \CC(\CC(X))$ is an embedding, see \cite{BG}.  By Ascoli's theorem \cite[3.4.20]{Eng}, each $k$-space is Ascoli. Moreover, Noble \cite{Noble} proved that any $k_\IR$-space is Ascoli.  We have the following implication
\[
k_\IR\mbox{-space}\Rightarrow \mbox{Ascoli},
\]
and this implication is not reversible (\cite{Banakh-Ascoli}).

The aforementioned results motivate the  following general  question.
\begin{question}
For which spaces $X$ and $Y$ the space $\CC(X,Y)$ is Ascoli?
\end{question}

Below we present the following  partial  answer to this question.

\begin{theorem} \label{t:Ascoli-Ck-Metriz}
For a metrizable space $X$, $\CC(X)$ is Ascoli if and only if $\CC(X)$ is a $k_\IR$-space if and only if $X$ is locally compact.
\end{theorem}

Corson \cite{Corson} started a systematic study of various topological properties of the weak topology of Banach spaces.  The famous Kaplansky Theorem states that a normed  space $E$ in the weak topology has countable tightness; for further results see \cite{edgar,GKZ}.  Schl\"{u}chtermann and Wheeler \cite{S-W} showed that an infinite-dimensional Banach space is never a $k$-space in the weak topology. We strengthen  this result as follows.
\begin{theorem} \label{t:Ascoli-Banach-weak}
A Banach space  $E$ in the weak topology is Ascoli if and only if $E$ is finite-dimensional.
\end{theorem}


Below we generalize Theorem \ref{t:Ascoli-Banach-weak} to an interesting class of Fr\'{e}chet locally convex spaces, i.e. metrizable and complete locally convex space (lcs). We say that a Fr\'echet lcs  $E$ is a \emph{quojection} if it is isomorphic to the projective limit of a sequence of Banach spaces with surjective linking maps or, equivalently,  if  every quotient of $E$ which
admits a continuous norm is a Banach space, see \cite{bellenot}. Obviously a countable product of Banach spaces is a quojection. Moscatelli \cite{moscatelli} gave examples of quojections which are not isomorphic to countable products of Banach spaces.
\begin{theorem} \label{t:Ascoli-quoj-weak}
Let a Fr\'{e}chet lcs $E$ be a quojection. Then $E$ in the weak topology is Ascoli if and only if $E$ is either finite-dimensional or is isomorphic to the product $\mathbb{K}^{\mathbb{N}}$, where $\mathbb{K}\in\{\mathbb{R},\mathbb{C}\}$.
\end{theorem}
 Since every Fr\'{e}chet lcs $\CC(X)$ is a quojection, see the survey \cite{bierstedt}, Theorem \ref{t:Ascoli-quoj-weak} yields the following
\begin{corollary}\label{c:Ascoli-Ck(X)-quoj}
For a Fr\'{e}chet lcs  $\CC(X)$,   the space $\CC(X)$ in the weak topology is Ascoli if and only if  $X$ is countable and discrete.
\end{corollary}

Let $E$ be a Banach space; denote by $B_w$ the closed unit ball $B=B_E$ in $E$ endowed with the weak topology of $E$. Schl\"{u}chtermann and Wheeler \cite{S-W} showed that some topological properties of $B_w$ are closely related to the isomorphic structure of $E$:

\begin{theorem}[\cite{S-W}] \label{t:S-W-1}
The following conditions for a Banach space $E$ are equivalent: {\em (a)} $B_w$ is Fr\'{e}chet--Urysohn; {\em (b)} $B_w$ is sequential; {\em (c)} $B_w$ is a $k$-space; {\em (d)} $E$ contains no isomorphic copy of $\ell_1$.
\end{theorem}

Therefore it  seems to be natural to verify  whether there exists a Banach space $E$ containing  a copy of $\ell_{1}$ and such that $B_{w}$ is Ascoli or a $k_\IR$-space. We answer such a question in the negative,  by proving  the following extension of Theorem \ref{t:S-W-1}. 

\begin{theorem}\label{t:Ascoli-Main}
Let $E$ be a Banach space and $B_{w}$ its closed unit ball with the weak topology. Then the following assertions are equivalent:
\begin{enumerate}
\item[{\rm (i)}]  $B_{w}$ is an Ascoli space;
\item[{\rm (ii)}]  $B_{w}$ is a $k_\IR$-space;
\item[{\rm (iii)}]  every sequentially continuous real-valued map on $B_{w}$ is continuous;
\item[{\rm (iv)}]  $E$ does not contain a copy of $\ell_{1}$.
\end{enumerate}
\end{theorem}

The proof of (i)$\Rightarrow$(iv) in Theorem \ref{t:Ascoli-Main}, given in  Proposition \ref{p:Ascoli-Main} below, uses  basic properties  of stochastically independent measurable functions. We also present a result related to Theorem \ref{t:Ascoli-Main} (ii), namely for Banach spaces containing an isomorphic copy of $\ell_1$ we provide, in a sense,  a canonical example of a sequentially continuous but not continuous function on $B_w$. Our construction builds on   measure-theoretic properties of $\ell_1$-sequences of continuous functions, see Example \ref{exa:Ascoli-1} below.

For Fr\'{e}chet lcs we supplement Theorem \ref{t:S-W-1}  by proving  the following theorem.
\begin{theorem} \label{t:Ascoli-E-weak}
For a Fr\'{e}chet lcs $E$ the following conditions are equivalent:
\begin{enumerate}
\item[{\rm (i)}] $E$ contains no isomorphic copy of $\ell_1$;
\item[{\rm (ii)}]  each closed and convex bounded subset of $E$ is Ascoli in the weak topology.
\end{enumerate}
\end{theorem}
Theorems \ref{t:Ascoli-Main}--\ref{t:Ascoli-E-weak}  heavily depend on  our result stating  that the closed unit ball $B$ of $\ell_1$ in the  weak topology is not an Ascoli space, see  Proposition \ref{p:Ascoli-L1-weak} below.


\section{The Ascoli property for function spaces. Proof of Theorem \ref{t:Ascoli-Ck-Metriz}} \label{sec:2}


We start from the definitions of the following well-known notions. A topological space $X$ is called
\begin{itemize}
\item[$\bullet$] {\em Fr\'{e}chet-Urysohn} if for any cluster point $a\in X$ of a subset $A\subset X$ there is a sequence $\{ a_n\}_{n\in\NN}\subset A$ which converges to $a$;
\item[$\bullet$] {\em sequential} if for each non-closed subset $A\subset X$ there is a sequence $\{a_n\}_{n\in\NN}\subset A$ converging to some point $a\in \bar A\setminus A$;
\item[$\bullet$] a {\em $k$-space} if for each non-closed subset $A\subset X$ there is a compact subset $K\subset X$ such that $A\cap K$ is not closed in $K$;
\item[$\bullet$] a {\em $k_\IR$-space} if  a real-valued function $f$ on $X$ is continuous if and only if its restriction $f|_K$ to any compact subset $K$ of $X$ is continuous.
\end{itemize}

Recall that the family of subsets
\[
[C;\epsilon]:= \{ f\in \CC(X): |f(x)|<\epsilon \; \forall x\in C\},
\]
where $C$ is a compact subset of $X$ and $\epsilon >0$, forms a basis of open neighborhoods at the zero function $\mathbf{0} \in\CC(X)$.
Below we give a simple sufficient condition on a space $X$  not to be Ascoli.
\begin{proposition} \label{p:Ascoli-sufficient}
Assume  a Tychonoff space $X$ admits a  family $\U =\{ U_i : i\in I\}$ of open subsets of $X$, a subset $A=\{ a_i : i\in I\} \subset X$ and a point $z\in X$ such that
\begin{enumerate}
\item[{\rm (i)}] $a_i\in U_i$ for every $i\in I$;
\item[{\rm (ii)}] $\big|\{ i\in I: C\cap U_i\not=\emptyset \}\big| <\infty$  for each compact subset $C$ of $X$;
\item[{\rm (iii)}] $z$ is a cluster point of $A$.
\end{enumerate}
Then $X$ is not an Ascoli space.
\end{proposition}

\begin{proof}
For every $i\in I$, take a continuous function $f_i: X\to [0,1]$ such that $f_i(a_i)=1$ and $f_i(X\setminus U_i)=\{ 0\}$. Set $\KK :=\{ f_i: i\in I\} \cup\{ \mathbf{0}\}$.

We claim that $\KK$ is a compact subset of $\CC(X)$ and $\mathbf{0}$ is a unique cluster point of $\KK$.
Indeed, let $C$ be a compact subset of $X$ and $\epsilon>0$. By (ii), the set $J:=\{ i\in I: C\cap U_i\not=\emptyset\}$ is finite. So, if $i\not\in J$, then $f_i(C)=\{ 0\}$. Hence $f_i\in [C;\epsilon]$ for every $i\in I\setminus J$. This means that $\KK$ is a compact set with the unique cluster point $\mathbf{0}$.

We show that $\KK$ is not evenly continuous considering $\mathbf{0}$, $z$ and $O=(-1/2,1/2)$. By the claim, any neighborhood $U_{\mathbf{0}}\subset \KK$ of $\mathbf{0}$ contains almost all functions $f_i$, and, by (iii), any neighborhood $O_z$ of $z$ contains infinitely many points $a_i$. So, there is $m\in I$ such that $f_m\in U_{\mathbf{0}}$ and $a_m\in O_z$. Since $f_m(a_m)=1$, we obtain that $U_{\mathbf{0}}(O_z)\not\subset O$. Hence $\KK$ is not evenly continuous. Thus $X$ is not Ascoli.
\end{proof}
The next corollary follows also from Proposition 5.11(1) of \cite{BG}.
\begin{corollary}\label{c:Ascoli-sufficient}
Let $X$ be a Tychonoff space with a unique cluster point $z$ and such that every compact subspace of $X$ is finite. Then $X$ is not an Ascoli space.
\end{corollary}
\begin{proof}
Since every $x\in X$, $x\not= z$, is isolated, we set $I=A=X\setminus\{ z\}$ and $U_x =\{x\}$ for $x\in A$. Now Proposition \ref{p:Ascoli-sufficient}  applies.
\end{proof}


The proof of the next proposition is a modification of the proof of the assertion in Section 5 of \cite{Pol-1974}.
\begin{proposition} \label{p:Ascoli-Pol}
Let $X$ be a first countable paracompact space. If $X$ is not locally compact, then $\CC(X)$  contains a countable family $\U =\{ U_s\}_{s\in\NN}$ of open subsets in $\CC(X)$ and a countable subset $A=\{ a_s\}_{s\in\NN} \subset \CC(X,\II)$ such that
\begin{enumerate}
\item[{\rm (i)}] $a_s\in U_s$ for every $s\in\NN$;
\item[{\rm (ii)}] if $K\subset \CC(X)$ is compact, the set $\{ s: U_s\cap K\}$ is finite;
\item[{\rm (iii)}] the zero function $\mathbf{0}$ is a cluster point of $A$.
\end{enumerate}
In particular, the spaces $\CC(X)$ and $\CC(X,\II)$ are not Ascoli.
\end{proposition}

\begin{proof}
Suppose for a contradiction that $X$ is not locally compact and let $x_0\in X$ be a point which does not have compact neighborhood. Take open  bases $\{ V'_i\}_{i\in\NN}$ and $\{ W_i\}_{i\in\NN}$ at $x_0$ such that
\[
V'_i \supset \overline{W_i} \supset W_i \supset \overline{V'_{i+1}}, \quad \forall i\in\NN.
\]
Set $P'_i :=\overline{V'_{i}}\setminus V'_{i+1}$, $\forall i\in\NN$. Since none of the sets $V'_i$ is compact, there exists a sequence $k_1 < k_2<\dots$ such that $P'_{k_i}$ is not compact and $k_{i+1}>k_i +1$. Set $P_i =P'_{k_i}$ and $V_i=V'_{k_i}$. Then $\{ P_i\}_{i\in\NN}$ is a sequence of closed, non-compact subsets of $X$, $\{ V_i\}_{i\in\NN}$ is a decreasing open  base at $x_0$ and
\begin{equation} \label{e:Ascoli-Ck-1}
P_i \subset \overline{V_i}\setminus \overline{W_{k_{i +1}}} \; \mbox{ and } \; \overline{V_{i+1}} \subset W_{k_{i +1}}.
\end{equation}

Fix arbitrarily $i\in\NN$. Since $P_i$ is not compact, by \cite[3.1.23]{Eng}, there is a one-to-one sequence $\{ x_{j,i}\}_{j\in\NN} \subset P_i$ which is discrete and closed in $X$. Now the paracompactness of $X$ and (\ref{e:Ascoli-Ck-1}) imply that there exists  an open sequence $\{ V_{j,i}\}_{j\in\NN}$ such that
\begin{equation} \label{e:Ascoli-Ck-2}
x_{j,i} \in V_{j,i}, \;  \mbox{ and } \; V_{j,i} \cap \overline{V_{i+1}} =\emptyset, \forall j\in \NN, \; \mbox{ and } \{ V_{j,i}\}_{j\in\NN} \mbox{ is discrete in } X.
\end{equation}

For every $p,q\in\NN$ such that $1\leq p<q$, choose continuous functions $f_{q,p}: X\to [0,1]$  such that
\begin{equation} \label{e:Ascoli-Ck-3}
f_{q,p} (x_{q,p})=1, \; f_{q,p} (x_{q,q})=0,  \; f_{q,p}(x_0)=1/p \; \mbox{ and } \; f_{q,p}(x)\leq 1/p \mbox{ for } x\not\in V_{q,p}.
\end{equation}
Set $A:=\{ f_{q,p}: 1\leq p<q<\infty \}$ and $\U =\{ U_{q,p}: 1\leq p<q<\infty\}$, where  $U_{q,p}$ is the set of all functions $h\in\CC(X)$ satisfying the inequalities
\begin{equation} \label{e:Ascoli-Ck-32}
\big| h(x_{q,p}) -1 \big| < \frac{1}{4^{p+q}}, \quad \big| h(x_0) -\frac{1}{p} \big| < \frac{1}{4^{p+q}}, \quad \big| h(x_{q,q}) \big| < \frac{1}{4^{p+q}}.
\end{equation}
Let us show that $A$ and $\U$ are as desired. Clearly, (i) holds. Let us prove (ii).

Fix a compact subset $K$ of $\CC(X)$.
Let us first observe that
\begin{equation} \label{e:Ascoli-Ck-4}
\mbox{ there exists $p_0\in\NN$ such that if $p\geq p_0$ and $q>p$, then $U_{q,p} \cap K=\emptyset$.}
\end{equation}
Indeed, otherwise we would find  sequences $p_1<q_1< p_2<q_2<\dots$ and $h_{q_i,p_i} \in U_{q_i,p_i} \cap K$. Set
\[
Z_1 :=\{ x_{q_i,p_i}: i\in\NN\} \cup\{ x_0\}.
\]
From (\ref{e:Ascoli-Ck-1}) it follows that $Z_1$ is compact, and thus, by the Ascoli Theorem \cite[3.4.20]{Eng}, there exists $r>10$ such that if $z',z''\in Z_1\cap \overline{V_r}$ and $f\in K$, then $|f(z')-f(z'')|<1/3$. But since $10< r\leq p_r <q_r$ we obtain $x_0, x_{q_r,p_r} \in Z_1\cap \overline{V_r}$. Hence, by  (\ref{e:Ascoli-Ck-32}), we have
\[
\big| h_{q_r,p_r}(x_{q_r,p_r}) -h_{q_r,p_r}(x_0)\big| > \left( 1- \frac{1}{4^{20}}\right) - \left( \frac{1}{p_r} + \frac{1}{4^{20}}\right) >\frac{1}{3}.
\]
Since $h_{q_r,p_r}\in K$, we get a contradiction.

We shall now prove that
\begin{equation} \label{e:Ascoli-Ck-5}
\mbox{ there exists $q_0\in\NN$ such that if $q\geq q_0$ and $1\leq p<p_0$, then $U_{q,p} \cap K=\emptyset$,}
\end{equation}
where $p_0$ is defined in (\ref{e:Ascoli-Ck-4}). Indeed, set
\[
Z_2 :=\{ x_{j,i} : 1\leq j\leq i<\infty \} \cup\{ x_0\}.
\]
Then $Z_2$ is compact by (\ref{e:Ascoli-Ck-1}). Again by the Ascoli Theorem, it follows that there exists $q_0\in\NN$ such that for $z',z''\in Z_2\cap \overline{V_{q_0}}$ and $f\in K$ we have $|f(z')-f(z'')|<1/4p_0$. The $q_0$ chosen in this way satisfies (\ref{e:Ascoli-Ck-5}), since otherwise there would exist $q\geq q_0$ and $1\leq p<p_0$ such that $U_{q,p} \cap K\not=\emptyset$. Fix $h_{q,p} \in U_{q,p} \cap K$. Then $x_0, x_{q,q} \in  Z_2\cap \overline{V_{q_0}}$, and by (\ref{e:Ascoli-Ck-3}) and (\ref{e:Ascoli-Ck-32}), we obtain
\[
\big| h_{q,p}(x_{q,q}) -h_{q,p}(x_0)\big| > \left( \frac{1}{p} -\frac{1}{4^{p+q}}\right) -\frac{1}{4^{p+q}} > \frac{1}{3p}> \frac{1}{4p_0},
\]
which gives a contradiction. Now  (\ref{e:Ascoli-Ck-4}) and (\ref{e:Ascoli-Ck-5}) immediately imply (ii).

Now we prove (iii).  Fix arbitrarily  a compact subset $Z\subset X$ and $\epsilon>0$. Choose $p_0$ such that $1/p_0 <\epsilon$.  By (\ref{e:Ascoli-Ck-2}), we can find $j_0\in\NN$ such that $Z\cap V_{j,p_0}=\emptyset$ for every $j\geq  j_0$. Take $q_0=p_0 +j_0$. Then $f_{q_0,p_0}\in A$, and for $z\in Z$ we have $z\not\in V_{q_0,p_0}$, and thus, in accordance with (\ref{e:Ascoli-Ck-3}), $f_{q_0,p_0}(z) \leq 1/p_0 <\epsilon$. Thus $f_{q_0,p_0}\in [Z;\epsilon]$.

Finally, the spaces $\CC(X)$ and $\CC(X,\II)$ are not Ascoli by Proposition \ref{p:Ascoli-sufficient}.
\end{proof}

The next corollary proved by R.~Pol  solves Problem 6.8 in \cite{BG}.
\begin{corollary}[\cite{Pol-2015}] \label{c:Ascoli-Ck-Pol}
For a separable metrizable space $X$, $\CC(X)$ is Ascoli if and only if $X$ is locally compact.
\end{corollary}
\begin{proof}
If $\CC(X)$ is Ascoli, then $X$ is locally compact by Proposition \ref{p:Ascoli-Pol}. Conversely, if $X$ is a separable metrizable locally compact space, then $\CC(X)$ is even a Polish space.
\end{proof}

Recall that a family $\mathcal{N}$ of subsets of a topological space $X$ is called a \emph{network} in $X$ if, whenever $x\in U$ with $U$ open in $X$, then $x\in N\subset U$ for some $N \in\mathcal{N}$. A space $X$ is called a {\em $\sigma$-space} if it is regular and has a $\sigma$-locally finite network. Any metrizable space is a $\sigma$-space by the Nagata-Smirnov Metrization Theorem.

Now Theorem \ref{t:Ascoli-Ck-Metriz} follows from the following theorem in which the equivalence of (i) and (ii) is well-known.
\begin{theorem} \label{t:Ascoli-Parac}
Let $X$ be  a first-countable paracompact $\sigma$-space. Then the following assertions are equivalent:
\begin{enumerate}
\item[{\rm (i)}] $X$ is a locally compact metrizable space;
\item[{\rm (ii)}] $X=\bigoplus_{i\in\kappa} X_i$, where all $X_i$ are separable metrizable locally compact spaces;
\item[{\rm (iii)}] $\CC(X)$ is a  $k_\IR$-space;
\item[{\rm (iv)}] $\CC(X)$ is an Ascoli space;
\item[{\rm (v)}] $\CC(X,\II)$  is a  $k_\IR$-space;
\item[{\rm (vi)}] $\CC(X,\II)$ is an Ascoli space.
\end{enumerate}
In cases (i)--(vi), the spaces $\CC(X)$  and $\CC(X,\II)$ are the products of families of Polish spaces.
\end{theorem}

\begin{proof}
(i)$\Rightarrow$(ii) follows from \cite[5.1.27]{Eng}.

(ii)$\Rightarrow$(iii),(v): If $X=\bigoplus_{i\in\kappa} X_i$, then
\[
\CC(X) =\prod_{i\in\kappa} \CC(X_i)\quad \mbox{ and } \quad \CC(X,\II) =\prod_{i\in\kappa} \CC(X_i,\II),
\]
 where all the spaces $ \CC(X_i)$ and $\CC(X_i,\II)$ are Polish (see Corollary \ref{c:Pol-k-space}). So $\CC(X)$ and $\CC(X,\II)$ are  $k_\IR$-spaces by \cite[Theorem 5.6]{Nob}.

(iii)$\Rightarrow$(iv) and (v)$\Rightarrow$(vi) follow from \cite{Noble}. The implications (iv)$\Rightarrow$(i) and (vi)$\Rightarrow$(i) follow from Proposition \ref{p:Ascoli-Pol} and the fact that any locally compact $\sigma$-space is metrizable by \cite{OMe}.
\end{proof}
Note that Theorem \ref{t:Ascoli-Parac} holds true for first-countable stratifiable spaces since any stratifiable space is a paracompact  $\sigma$-space (see Theorems 5.7 and 5.9 in \cite{gruenhage}).


\section{Proofs of Theorems \ref{t:Ascoli-Banach-weak} and \ref{t:Ascoli-quoj-weak}}


Following Arhangel'skii \cite[II.2]{Arhangel}, we say  that a topological space $X$ has {\em countable fan tightness at a point } $x\in X$ if for each sets $A_n\subset X$, $n\in\NN$, with $x\in \bigcap_{n\in\NN} \overline{A_n}$ there are finite sets $F_n\subset A_n$, $n\in\NN$, such that $x\in \overline{\cup_{n\in\NN} F_n}$; $X$ has {\em countable fan tightness} if $X$ has countable fan tightness at each point $x\in X$. Clearly, if $X$ has countable fan tightness, then $X$ also has countable tightness.

For a topological space $X$  we denote by $C_p(X)$ the space $C(X)$ endowed with the topology of poitwise convergence.

For a lcs $E$, denote by $E'$ the dual space of $E$. The space $E$ endowed with the weak topology $\sigma(E,E')$ is denoted by $E_w$. The closure of a subset $A\subset E$ in $\sigma(E,E')$ we denote by $\overline{A}^{\, w}$. If $E$ is a metrizable lcs, then $X:=(E', \sigma(E',E))$ is $\sigma$-compact by the Alaoglu--Bourbaki Theorem. Since $E_w$ embeds into $C_p(X)$, Theorem II.2.2 of \cite{Arhangel} immediately implies the following result noticed in \cite{GKZ}.
\begin{fact}[\cite{GKZ}]\label{t:Weak-Fan-tight}
If $E$ is a metrizable lcs, then $E_w$ has countable fan tightness.
\end{fact}

Denote the unit sphere of a normed space $E$  by $S_E$. Theorem \ref{t:Ascoli-Banach-weak} immediately follows from the next proposition.
\begin{proposition}\label{p:Ascoli-normed-weak}
Let $E$ be a normed space. Then $E$ with the weak topology is Ascoli if and only if $E$ is finite-dimensional.
\end{proposition}
\begin{proof}
We show that $E_w$ is not Ascoli for any infinite-dimensional normed space $E$.

For every $n\in\NN$, let $A_n$ be a countable subset of $nS$ such that $0\in \overline{A_n}^{\, w}$
(see \cite[Exercise 3.46]{fabian-10} and Fact \ref{t:Weak-Fan-tight}). Now Fact \ref{t:Weak-Fan-tight} implies that there are finite sets  $F_n\subset A_n$, $n\in\NN$, such that $0\in \overline{\cup_{n\in\NN} F_n}$. Set $A:= \bigcup_{n\in\NN} F_n$. Using the Hahn--Banach Theorem, for every $n\in\NN$ and each $a\in F_n$ take a weakly open neighborhood $U_a$ of $a$ such that
\begin{equation} \label{e:Ascoli-normed-weak}
U_a \cap \left(n-\frac{1}{2} \right)B =\emptyset.
\end{equation}
Let us show that the family $\U=\{ U_a: a\in A\}$, the set $A$ and the zero $0$ satisfy conditions (i)--(iii) of Proposition \ref{p:Ascoli-sufficient}. Clearly, (i) and (iii) hold. To check (ii), let $C$ be a compact subset of $E_w$. Then $C\subset mB$ for some $m\in\NN$, and (\ref{e:Ascoli-normed-weak}) implies that the set
\[
\{ a\in A: U_a \cap C \not=\emptyset \} \subset \bigcup_{n\leq m} F_n
\]
is finite. Finally, Proposition \ref{p:Ascoli-sufficient} implies that $E_w$ is not Ascoli. $\Box$
\end{proof}
We need also the following
\begin{proposition} \label{p:Ascoli-Quotient}
Let $p:X\to Y$ be an open continuous  map of a topological space $X$ onto a regular space $Y$. If $X$ is Ascoli, then $Y$ is also an Ascoli space.
\end{proposition}

\begin{proof}
 Let $\KK$ be a compact subset of $\CC(Y)$. We have to show that $\KK$ is evenly continuous. Denote by $p^\ast : \CC(Y)\to \CC(X), p^\ast(h):=h(p(x))$, the adjoin continuous map.

 Fix $y_0\in Y$, $h_0\in \KK$ and an open neighborhood $O_{z_0}$ of the point $z_0 :=h_0(y_0)$. Set $f:=p^\ast(h_0)\in \CC(X)$ and take arbitrarily a preimage $x_0$ of $y_0$, so $p(x_0)=y_0$. Since $p^\ast(\KK)$ is a compact subspace of $\CC(X)$ it is evenly continuous. Hence we can find neighborhoods $U_{f}\subset  p^\ast(\KK)$ of $f$ and $O_{x_0} \subset X$ of $x_0$ such that $U_f (O_{x_0}) \subset O_{z_0}$. Set   $U_{h_0} := \KK\cap \big( p^\ast\big)^{-1} (U_f)$ and $O_{y_0} := p\big( O_{x_0}\big)$ (which is a neighborhood of $y_0$ as $p$ is open). For every $h\in U_{h_0}$ and each $y\in O_{y_0}$, take $x\in O_{x_0}$ with $p(x)=y$, so we obtain
\[
h(y)=h(p(x))=p^\ast(h)(x) \in O_{z_0}.
\]
Thus $\KK$ is evenly continuous, and therefore $Y$ is Ascoli.
\end{proof}

Below we prove Theorem  \ref{t:Ascoli-quoj-weak} and Corollary \ref{c:Ascoli-Ck(X)-quoj}.

{\em Proof of Theorem \ref{t:Ascoli-quoj-weak}.}
Assume that $E$ is infinite-dimensional. By Proposition \ref{p:Ascoli-normed-weak} the space $E$ is not normed. Let $(p_{n})_{n}$ be a sequence of continuous seminorms providing the topology of $E$. For each $n\in\mathbb{N}$, let $E_{n}:=E/p_{n}^{-1}(0)$ be the quotient endowed with the norm topology $p_{n}^{*}: [x]\mapsto p_{n}(x)$, where $[x]$ is the equivalence class of $x$ in $E$. Since $E$ is a quojection, the quotient $E_{n}$ with the original quotient topology is a Banach space by \cite[Proposition 3]{bellenot}.

By Proposition \ref{p:Ascoli-Quotient} the space $E_{n}$ endowed with the weak topology is Ascoli, so we apply  Proposition \ref{p:Ascoli-normed-weak} to deduce that each $E_{n}$ is finite-dimensional. On the other hand, $E$ embeds into the product $\prod_{n} E_{n}$. So $E$, being complete, is isomorphic to  a closed subspace of the product $\mathbb{K}^{\mathbb{N}}$. Thus $E$ is also isomorphic to $\mathbb{K}^{\mathbb{N}}$ by \cite[Corollary 2.6.5]{bonet}.
$\Box$

{\em Proof of Corollary \ref{c:Ascoli-Ck(X)-quoj}.}
By Theorem \ref{t:Ascoli-quoj-weak} the space $\CC(X)$ is isomorphic to $\mathbb{R}^{\mathbb{N}}$, and since $\mathbb{R}^{\mathbb{N}}$ does not admit a weaker locally convex topology (see \cite[Corollary 2.6.5]{bonet}), $\CC(X)=C_{p}(X)=\mathbb{R}^{\mathbb{N}}$. Thus $X$ is countable and discrete. The converse assertion is trivial.
$\Box$

We do not know whether there exists a Fr\'{e}chet space $E$ such that $E_w$ is an Ascoli non-metrizable space.
\begin{remark}\label{kothe} {\em
The first example of a non-distinguished Fr\'echet space (so also not quojection) was given by Grothendieck and K\"{o}the, and it
was the K\"{o}the echelon space $\lambda_{1}(A)$  of order $1$ for the K\"{o}the matrix $A = (a_{n})_{n}$  defined on $\mathbb{N}\times\mathbb{N}$ by
$a_{n}(i,j):= j$ if $i < n$ and $a_{n}(i, j)=1$ otherwise, see \cite{bierstedt} also for more references.  We do not know however if  this space with the weak topology is an Ascoli space. }
\end{remark}



\section{Proof of Theorem \ref{t:Ascoli-Main}}


To prove Theorem \ref{t:Ascoli-Main} we need the following key proposition,  which proves, among others,  that the unit ball $B_{\ell_1}$ in the weak topology is not Ascoli. In particular, since the $k$-space property is inherited by the closed subspaces, this shows also that any Banach space $E$ whose weak unit ball $B_w$ is a $k$-space contains no isomorphic copy of $\ell_1$, i.e. the proposition proves (c)$\Rightarrow$(d) in  Schl\"{u}chtermann--Wheeler's theorem \ref{t:S-W-1}. A sequence $\{ x_i\}_{i\in\NN} \subset E$ is called {\em trivial} if there is $n\in\NN$ such that $x_i=x_n$ for all $i>n$.

\begin{proposition} \label{p:Ascoli-L1-weak}
Let $E=\ell_1$ and $B_w$ its closed unit ball in the weak topology. Then there is a countable subset $A$ of $S_{\ell_1}$ and a family $\U =\{ U_a: a\in A\}$ of weakly open subsets of the unit ball $B$ such that
\begin{enumerate}
\item[{\rm (1)}] $a\in U_a$ for every $a\in A$;
\item[{\rm (2)}] $\mathrm{dist}(U_a, U_b)\geq 1/5$ for every distinct $a,b\in A$;
\item[{\rm (3)}] the zero $0$ is the unique cluster point of $A$;
\item[{\rm (4)}] $\big|\{ a\in A: C\cap U_a\not=\emptyset \}\big| <\infty$ for every weakly compact subset $C$ of $B$;
\item[{\rm (5)}] $\overline{A}^{\, w}=A\cup\{ 0\}$ and every weakly compact subset of $\overline{A}^{\, w}$ is finite;
\item[{\rm (6)}] $A$ contains a sequence which is equivalent to the unit basis of $\ell_1$;
\item[{\rm (7)}] the set $A$ does not have a non-trivial weakly fundamental subsequence;
\item[{\rm (8)}] the countable space $\overline{A}^{\, w}$  and $B_w$ are not Ascoli.
\end{enumerate}
\end{proposition}

\begin{proof}
Let $\{(e_i,e^\ast_i): i\in\NN\}$ be the standard biorthogonal basis in $\ell_1\times \ell'_1=\ell_1\times \ell_\infty$. Following \cite{GKZ}, set $\Omega :=\{ (m,n)\in \NN\times\NN : m< n\}$ and
\[
A:= \left\{ a_{m,n}:= \frac{1}{2}(e_m - e_n): (m,n)\in\Omega \right\} \subset S_{\ell_1}.
\]
For every $(m,n)\in\Omega$, define the following weak neighborhood of $a_{m,n}$
\[
\begin{split}
U_{m,n} & := \left\{ x\in B: \ |\left\langle e^\ast_m, a_{m,n} -x\right\rangle |< \frac{1}{10} \mbox{ and } |\left\langle e^\ast_n, a_{m,n} -x\right\rangle |< \frac{1}{10} \right\}\\
 & \; =  \left\{ x=(x_i) \in B: \ \left| \frac{1}{2} -x_m \right| <\frac{1}{10}  \mbox{ and } \left| \frac{1}{2} +x_n \right| <\frac{1}{10}\right\}.
\end{split}
\]
Then (1) holds trivially. Let us check (2). For every $k\not\in \{ m,n\}$ and each $x=(x_i)\in U_{m,n}$, one has
\[
|x_k|\leq \| x\| - |x_m| -|x_n| < 1- \left(\frac{1}{2}-\frac{1}{10}\right) -\left(\frac{1}{2}-\frac{1}{10}\right)=\frac{1}{5}.
\]
So, if $(m,n)\not= (k,l)$ and $x=(x_i)\in U_{m,n}$, we obtain either
\[
\left| \frac{1}{2} -x_k \right| >\frac{1}{2}-\frac{1}{5}=\frac{3}{10} \mbox{ if } k\not\in \{ m,n\}, \quad \mbox{ or }\quad \left| \frac{1}{2} +x_l \right| >\frac{3}{10} \mbox{ if } l\not\in \{ m,n\}.
\]
Hence $\mathrm{dist}\big( U_{m,n}, U_{k,l}\big) \geq 3/10 - 1/10=1/5$ for all $(m,n)\not= (k,l)$. This proves (2). In particular,  every point of $A$ is weakly isolated.

To prove (3) we note first that $0\in \overline{A}^{\, w}$ by Lemma 3.2 of \cite{GKZ}.
We provide a proof of this result to keep the paper self-contained. Let $U$ be a neighborhood of $0$ of the canonical form
\[
U=\left\{ x\in\ell_1 : |\langle\chi_k,x\rangle|<\epsilon , \mbox{ where } \chi_k =\big( \chi_k(i)\big)_{i\in\NN} \in S_{\ell_\infty} \mbox{ for }  1\leq k\leq s\right\}.
\]
Let $I$ be  an infinite subset of $\NN$ such that, for every $1\leq k\leq s$, either $\chi_k(i)>0$  for all $i\in I$, or $\chi_k(i)=0$  for all $i\in I$, or $\chi_k(i)<0$  for all $i\in I$. Take a natural number $N> 1/\epsilon$. Since $I$ is infinite, by induction,  one can find $(m,n)\in\Omega$ satisfying the following condition:  for every $1\leq k\leq s$ there is $0< t_k\leq N$ such that
\begin{equation}\label{equ:L1-1}
\frac{t_k -1}{N} \leq \min\big\{ |\chi_k(m)| , |\chi_k(n)| \big\} \leq \max\big\{ |\chi_k(m)| , |\chi_k(n)| \big\} \leq \frac{t_k}{N}.
\end{equation}
Then, by the construction of $I$, we obtain
\[
|\langle\chi_k, a_{m,n} \rangle|<1/N <\epsilon \mbox{ for every } 1\leq k\leq s.
\]
Thus $a_{m,n}\in U$, and hence $0\in \overline{A}^{\, w}$.

Now fix arbitrarily a nonzero $z=(z_i)\in\ell_1$ and consider the following three cases.

(a) There is $z_i\not\in \{ -1/2, 0, 1/2\}$, so $z\not\in A$. Set
\[
\epsilon :=\frac{1}{2}\min\left\{ |z_i|, \left|z_i -\frac{1}{2}\right|, \left|z_i +\frac{1}{2}\right|\right\} \mbox{ and } U:=\{ x\in\ell_1 : |\langle e^\ast_i, z-x\rangle| <\epsilon \}.
\]
Clearly, $U\cap A=\emptyset$ and $z\not\in\overline{A}^{\, w}$.

(b) Assume that $z\not\in A$ and  $z_i\in \{ -1/2, 0, 1/2\}$ for every $i\in\NN$. So there are distinct indices $i$ and $j$ such that $z_i =z_j \in\{ -1/2, 1/2\}$.  Set
\[
U:=\{ x\in\ell_1 : |\langle e^\ast_i +e^\ast_j, z-x\rangle| <1/10 \}.
\]
By the definition of $A$, we obtain  $U\cap A=\emptyset$, and hence $z\not\in\overline{A}^{\, w}$.

(c) Assume that $z\in A$. Then $z$ is not a cluster point of $A$ because it is weakly isolated.

Now (a)--(c) prove (3). Let us prove (4). Fix a weakly compact subset $C$ of $\ell_1$.
Assuming that  $C\cap U_a\not=\emptyset$ for an infinite subset $J\subset A$ we choose $x_j\in C\cap U_j$ for every $j\in J$. Since $\ell_1$ has the Schur property, $C$ is also compact in the norm topology of $\ell_1$. So we can assume that $x_j$ converges to some $x_\infty\in C$ in the norm topology. But this contradicts (2) that proves (4).

(5) immediately follows from (3) and (4).

(6): Clearly, the sequence $\{ a_{1,i}\}_{i>1} \subset A$ is equivalent to the unit basis of $\ell_1$.

(7): Assuming the converse let  $\{ a_{m_i,n_i}\}_{i\in\NN}$ be a faithfully indexed  weakly fundamental subsequence of $A$. Then only the next two cases are possible.

{\em Case 1. There is $k\in\NN$ and $i_1 < i_2<\dots$ such that $k=m_{i_1}=m_{i_2}=\dots $.} Passing to a subsequence we can assume that $m_1=m_2=\dots =k$ and $k<n_1<n_2<\dots$. Set
\[
\chi :=(\chi_j)_{j\in\NN}\in\ell_\infty, \mbox{ where } \chi_j =\left\{
\begin{split}
-1, & \mbox{ if } j\in\{ n_2, n_4,\dots\},\\
0, & \mbox{ if } j\not\in \{ n_2, n_4,\dots\}.
\end{split} \right.
\]
Then $\chi\in S_{\ell_\infty}$ and
\[
\langle\chi, a_{k,n_{2s}} - a_{k,n_{2s+1}}\rangle =\frac{1}{2}, \quad \forall s\in\NN.
\]
Thus the sequence $\{ a_{m_i,n_i}\}_{i\in\NN}$ is not fundamental, a contradiction.

{\em Case 2. $m_i\to\infty$ and $n_i\to\infty$.} Passing to a subsequence if it is needed, we can assume that
\[
m_1<n_1 <m_2< n_2<\dots .
\]
Defining $\chi\in S_{\ell_\infty}$ as in Case 1, we obtain
\[
\langle\chi, a_{m_{2s},n_{2s}} - a_{m_{2s+1},n_{2s+1}}\rangle =\frac{1}{2}, \quad \forall s\in\NN.
\]
Thus the sequence $\{ a_{m_i,n_i}\}_{i\in\NN}$ is not  weakly fundamental  also in this case.

Therefore $A$ does not have a  weakly fundamental subsequence.

(8):  The space $\overline{A}^{\, w}$ is not Ascoli by (5) and Corollary \ref{c:Ascoli-sufficient}, and $B_w$ is not Ascoli by (1)-(4) and Proposition \ref{p:Ascoli-sufficient}.
\end{proof}







Recall that a (normalized) sequence $(x_n)$ in a Banach space $E$ is said to be equivalent to the standard basis of $\ell_1$, or simply called an $\ell_1$-sequence, if for some $\theta>0$
\[
\left\| \sum_{i=1}^n c_i x_i\right\|\geq \theta\cdot \sum_{i=1}^n |c_i|,
\]
for any natural number $n$ and any scalars $c_i\in\mathbb{R}$. We also call such a sequence  a $\theta$-$\ell_1$-sequence if we want to specify the constant
in the definition.

We need some measure-theoretic preparations. Let $(T,\Sigma,\mu)$ be a probability measure space. Measurable functions $g_n: T\to\er$ are said to be {\em stochastically independent with respect to} $\mu$ if
\[
\mu\left( \bigcap_{n\leq k} g_n^{-1}(B_n)\right)=\prod_{n\leq k} \mu\left(g_n^{-1}(B_n)\right),
\]
for every $k$ and any Borel sets $B_n\sub\er$; see e.g.\ Fremlin \cite[272]{MT2}, for basic facts concerning independence.  Recall (see \cite[272Q]{MT2}) that, if integrable functions $f,g:T\to\er$ are independent with respect to $\mu$, then $\int_T f\cdot g\dx\mu=\left(\int_Tf\dx\mu\right) \cdot \left(\int_Tg \dx\mu\right)$.
\begin{lemma}\label{l:Ascoli-Im-mu}
Let $(T,\Sigma,\mu)$ and $(S,\Theta,\nu)$ be probability measure spaces and let $\Phi:T\to S$ be a measurable mapping such that $\Phi[\mu]=\nu$, that is $\mu(\Phi^{-1}(E)) =\nu(E)$ for every $E\in\Theta$. If $(p_n)_n$ be a sequence of measurable functions $S\to\IR$ which is stochastically independent with respect to $\nu$, then the functions $g_n=p_n\circ \Phi$ are stochastically independent with respect to $\mu$.
\end{lemma}
Lemma \ref{l:Ascoli-Im-mu} is standard and follows for instance from Theorem 272G in   \cite{MT2}.

In the proof of crucial Proposition \ref{p:Ascoli-Main} we essentially use the following version of the Riemann-Lebesgue lemma, which is mentioned in Talagrand's \cite{Ta}, page 3.

\begin{theorem} \label{t:Ascoli-Talagr}
Let $(T,\Sigma,\mu)$ be any probability space and let $(g_n)_n $ be a stochastically independent uniformly bounded sequence of measurable functions $T\to\er$ with
$\int_T g_n\dx\mu=0$ for every $n$. Then
\[\lim_{n\to\infty} \int_T f\cdot g_n\dx\mu=0,\]
for every bounded measurable function $f:T\to\er$.
\end{theorem}

Finally, let us recall the following fact, see e.g.\ \cite{Ta}, 1-2-5.

\begin{lemma}\label{image_measure}
Let $\Phi$ be a continuous surjection of a compact space $K$ onto a compact space $L$.
If $\lambda$ is a regular probability Borel measure on $L$ then there exists a regular probability Borel measure $\mu$ on $K$ such that $\Phi[\mu]=\lambda$, that
is $\mu(\Phi^{-1}(B))=\lambda(B)$ for every Borel set $B\sub L$.
\end{lemma}

\begin{proposition} \label{p:Ascoli-Main}
If a Banach space $E$ contains an isomorphic copy of $\ell_1$, then $B_w$ is not an Ascoli space.
\end{proposition}

\begin{proof}
We show that $B_w$ is not Ascoli in four steps.

{\em Step 1}. Since the Hilbert cube $H=[0,1]^\NN$ is separable, one can find a continuous function $\Phi_0$ from the discrete space $\NN$ onto a dense subset of $H$. By Theorem 3.6.1 of \cite{Eng}, we can extend $\Phi_0$ to a continuous map $\Phi:\beta\NN \to H$. As $\Phi_0(\NN)$ is dense in $H$, we obtain that $\Phi(\beta\NN) =H$.
Let $\pi_n :H\to [-1,1]$ be the projection onto the $n$th coordinate, and let $\lambda=\prod_n \mathrm{m}_n$ be  the product measure of the normalized Lebesgue measures $\mathrm{m}_n$ on the interval $[-1,1]$. Then the sequence $(\pi_n)$ is stochastically independent with respect to $\lambda$   and
\begin{equation} \label{equ:Ascoli-Ascoli-1}
\int_H \pi_n\dx\lambda =\int_H \pi_n \pi_m \dx\lambda =0, \mbox{ and } \int_H \pi_n^2\dx\lambda=\frac{1}{2}\int_{-1}^1 x^2 \dx x =\frac{1}{3},
\end{equation}
for all $n,m\in\NN$ and $n\not= m$. Moreover, the sequence $(\pi_n)_n$ is a $1\mbox{-}\ell_1$-sequence in $C(H)$. Indeed, for every $n\in\NN$ and each scalars $c_1,\dots,c_n\in\IR$, set
\[
x:=\big(\mathrm{sign}(c_1),\dots,\mathrm{sign}(c_n),0,\dots\big)\in H.
\]
Then $\sum_{i\leq n} c_i \pi_i(x) = \sum_{i\leq n} |c_i|$. Thus $(\pi_n)$ is a $1\mbox{-}\ell_1$-sequence in $C(H)$.

{\em Step 2.}
Let $\mu$ be a measure on $\beta\NN$ such that $\Phi[\mu]=\lambda$, see Lemma \ref{image_measure}. Set $g_n :=\pi_n\circ \Phi$ for every $n\in\NN$. Then the sequence $(g_n)$ is stochastically independent with respect to $\mu$ by Lemma \ref{l:Ascoli-Im-mu}. As $\Phi$ is surjective, $(g_n)$ is also a $1\mbox{-}\ell_1$-sequence in $C(\beta\NN)$.

{\em Step 3.}  Let $Y$ be a subspace of $E$ isomorphic to $\ell_1$ and let $T_1:Y\to \ell_1$ be an isomorphism. For every $n\in\NN$ choose $x_n\in Y$ such that $T_1(x_n)=e_n$, where $(e_n)$ is the standard coordinate basis in $\ell_1$. In turn, as $(g_n)$ is a $1\mbox{-}\ell_1$-sequence in $C(\beta\NN)$, there is an isometric embedding
$T_2:\ell_1\to C(\beta\NN)$, sending $e_n$ to $g_n$.

As the space $C(\beta\NN)$ is 1-injective, the operator $T=T_2\circ T_1:Y\to C(\beta\NN)$ can be extended to an operator $\widetilde{T}:E\to C(\beta\NN)$ having the same norm;
cf.\  Proposition 5.10 of \cite{fabian-10}.

{\em Step 4.} Set $d:=\sup\{ \| x_n\|_E : n\in\NN\}$ and $\gamma :=\sup\{ \| \widetilde{T}(x)\| : x\in d B_E\}$. Let $h_{m,n}=(g_m-g_n)/2$ for $n,m\in\NN, n>m,$ and set
\[
V_{m,n}=\left\{ f\in \gamma B_{C(\beta\NN)} : \left| \int_{\beta\NN} f \cdot g_i\dx\mu\right| >1/4, \mbox{ for } i=m,n\right\}.
\]
Denote by $T^+$ the map $\widetilde{T}$ from $E_w$ into $C_w(\beta\NN)$.  Clearly, $T^+$ is also continuous. Finally we set
\[
A:=\{ a_{m,n}:=(x_m-x_n)/2 : \; 1\leq m<n \},
\]
and
\[
\U := \{ U_{m,n} :=(T^+)^{-1} (V_{m,n}) \cap dB_E : \; 1\leq m<n \}.
\]
Now the following claim finishes the proof.

{\bf Claim}. {\em The ball $dB_E$ is not Ascoli in the weak topology.}

To prove the claim it is enough to check (i)-(iii) of Proposition \ref{p:Ascoli-sufficient} for the set $A$ and the family $\U$.

(i): To  show that $a_{m,n}\in U_{m,n}$ it is enough to prove that $h_{m,n}\in V_{m,n}$. But this follows from  (\ref{equ:Ascoli-Ascoli-1}) since
\[
2\int_{\beta\NN} h_{m,n}\cdot g_n\dx\mu =\int_{\beta\NN} g_m\cdot g_n\dx\mu -\int_{\beta\NN} g_n^2\dx\mu=-\frac{2}{3} = -2\int_{\beta\NN} h_{m,n}\cdot g_m\dx\mu .
\]
(iii):  The zero function $\mathbf{0}$ is the weak cluster point of $A$ by Proposition \ref{p:Ascoli-L1-weak}.

Let us check (ii), i.e. if $C\sub dB_E$ is weakly compact, then $C$ can meet only finite number of $U_{m,n}$'s.
Suppose otherwise: let $x_i \in C\cap U_{m_i,n_i}$, where the pairs $(m_i, n_i)$ are distinct. As $m_i <n_i$ we may assume also that $n_i\neq n_{i'}$ for $i\neq i'$.  Since $C$ is weakly compact it is Fr\'{e}chet--Urysohn by the Eberlein--\v{S}mulyan theorem \cite[3.109]{fabian-10}. So we can further assume that $x_i$ converge weakly to some $x\in C$.  Then also the functions $f_i :=T^+(x_i)\in V_{m_i, n_i}$ converge weakly to $f:=T^+(x)\in T^+(C)\subset \gamma B_{C(\beta\NN)}$, and they are uniformly bounded on $\beta\NN$ and $f_i\to f$ pointwise.

Take arbitrarily $0<\delta <1/16(1+ \gamma +2\gamma^2)$. By Theorem \ref{t:Ascoli-Talagr}, there is $N_1\in\NN$ such that  $\left| \int_{\beta\NN} f\cdot g_{n_i}\dx\mu\right| <\delta$ for all $i> N_1$.  By the classical Egorov theorem,
$f_i$ converge almost uniformly to $f$, i.e. there is $B\subseteq \beta\NN$ such that $\mu(\beta\NN\setminus B)<\delta$ and $f_i$ converge uniformly to $f$ on $B$. Take $N_2 >N_1$ such that $|f_i -f|< \delta$ on $B$ for all $i>N_2$. Taking into account that $|h|\leq\gamma$ for each $h \in \gamma B_{C(\beta\NN)}$, for every $i>N_2$ we obtain
\[
\begin{split}
\left| \int_{\beta\NN} f_i \cdot g_{n_i}\dx\mu \right| & \leq \left| \int_{\beta\NN} f_i\cdot g_{n_i}\dx\mu - \int_{\beta\NN} f \cdot g_{n_i}\dx\mu \right| + \left|\int_{\beta\NN} f \cdot g_{n_i}\dx\mu \right| \\
& \leq \int_{\beta\NN} |f_i-f|\cdot |g_{n_i}|\dx\mu + \delta \leq \int_B + \int_{\beta\NN \setminus B} + \delta \\
& \leq  \gamma\cdot \delta + 2\gamma^2 \cdot \delta + \delta =\delta (1+ \gamma +2\gamma^2) < 1/16.
\end{split}
\]
On the other hand, $f_i\in V_{m_i,n_i}$ implies $\left| \int_{\beta\NN} f_i\cdot g_{n_i}\dx\mu\right| >1/4$. This contradiction proves the claim.
\end{proof}

To prove Theorem \ref{t:Ascoli-Main} we need also the following simple lemma.

\begin{lemma}\label{l:Ascoli-sequential}
Let $E$ be a Banach space and let $B_w$ denote  the unit ball of $E$  equipped with the weak topology. For any function $f:B_w\to\IR$ the following are equivalent
\begin{itemize}
\item[{\rm (i)}] $f$ is sequentially continuous on $B_w$;
\item[{\rm (ii)}]  $f$ is continuous on every compact subset of $B_w$.
\end{itemize}
\end{lemma}

\begin{proof}
Let $f$ be sequentially continuous on $B_w$ and let $C$ be a compact subset of $B_w$.  For any closed set $H\sub\IR$,
the set $F=f^{-1}(H)\cap C$ is sequentially closed in $C$.  Hence $F$ is closed in $C$,  since  $C$, as  a weakly compact set, has the Frechet--Urysohn property  by the classical Eberlein--\v{S}mulian theorem.

We have checked that (i) imples (ii); the reverse implication is obvious.
\end{proof}

\begin{proof}[Proof of Theorem \ref{t:Ascoli-Main}]
(i)$\Rightarrow$(iv) follows from Proposition \ref{p:Ascoli-Main}. Theorem \ref{t:S-W-1} implies (iv)$\Rightarrow$(iii).  (iii)$\Rightarrow$(ii) follows from Lemma \ref{l:Ascoli-sequential}. Finally, the implication (ii)$\Rightarrow$(i) holds by \cite{Noble}.
\end{proof}


\section{On weakly sequentially continuous functions on the unit ball}


 Let $E$ be a Banach space containing an isomorphic copy of $\ell_1$ and let $B_w$ denote the unit ball in $E$ equipped with the weak topology. It follows from Theorem \ref{t:Ascoli-Main}  that $B_w$ is not a $k_\IR$-space which,  in view of Lemma \ref{l:Ascoli-sequential}, is equivalent to saying that there is a function $\Phi:B_w\to\er$ which is sequentially continuous but not continuous. We show below that such a function can be defined, in a sense, effectively by means of measure-theoretic properties of $\ell_1$-sequences of continuous functions.

 \begin{proposition} \label{2:1}
Let $K$ be a  compact space and let $(g_n)$ be a normalized $\theta\mbox{-}\ell_1$-sequence in the Banach space $C(K)$. Then there exists a regular probability measure $\mu$ on $K$ such that
\[\int_K |g_n-g_k| \dx\mu  \ge \theta/2  \mbox{ whenever } n\neq k.\]
\end{proposition}

\begin{proof}
Suppose that $(g_n)$ is $\theta$-equivalent to the standard basis $(e_n)$ in $\ell_1$. Put
\[
H=\overline{ {\rm conv} }\left( \{|g_n-g_k|: n\neq k\}\right) \sub C(K).
\]
Note that it is enough to check that $\|h\|\geq \theta/2$ for all $h\in H$ since in such a case, by the separation theorem,  there is a norm-one
$\mu\in C(K)^*$ such that $\int_K h\dx\mu\ge \theta/2$ for every $h\in H$. As $h\geq 0$ for $h\in H$, we can then replace the signed measure $\mu$ by its
variation $|\mu|$.

In turn,  the fact  that $\| h \|\geq \theta/2$  for $h\in H$ is implied by  the following.
\medskip

{\bf Claim}.  Suppose that $n_i\neq k_i$ for $i\le p$. then  for  any convex coefficients $\alpha_1,\ldots, \alpha_p$
\[
\left\| \sum_{i=1}^p  \alpha_i |g_{n_i}-g_{k_i}|  \right\| \geq\theta/2.
\]
\medskip

We shall verify the claim in two steps.

{\em Step 1.}
There is   $E\sub \{1,\ldots, p\}$ such that
\[
\left\| \sum_{i\in E} \alpha_i (e_{n_i} - e_{k_i})\right\|\geq 1/2.
\]

Indeed, if $L$ denotes the Cantor set $\{-1,1\}^\NN$, then  the projections $\pi_n:L\to \{-1,1\}$ form a  sequence in $C(L)$ which is a $1$-$\ell_1$-sequence,
 so  we have an isometric embedding $T:\ell_1\to C(L)$, where $Te_n=\pi_n$ for every $n\in\NN$.

Write $\lambda$ for the standard product measure on $L$. We calculate directly that $\int_K |\pi_n-\pi_k| \dx\lambda = 1$ for $n\neq k$ and therefore
\[
\left\| \sum_{i=1}^p \alpha_i|\pi_{n_i}-\pi_{k_i}|\right\| \geq \int_L \sum_{i=1}^p \alpha_i  |\pi_{n_i}-\pi_{k_i}| \dx\lambda=1.
\]
Hence there is $t\in L$ such that $\sum_{i=1}^p\alpha_i |\pi_{n_i}(t)-\pi_{k_i}(t)|\ge 1$. Examining the signs of summands we conclude that for some set $E\sub \{1,\ldots, p\}$ we have
\[
\left| \sum_{i\in E}\alpha_i (\pi_{n_i}(t)-\pi_{k_i}(t)) \right|\ge 1/2.
\]
This implies that
\[
\left\| \sum_{i\in E} \alpha_i (e_{n_i} - e_{k_i})\right\|= \left\| \sum_{i\in E} \alpha_i (Te_{n_i} -Te_{k_i}) \right\|\geq
\left| \sum_{i\in E } \alpha_i (\pi_{n_i}(t)-\pi_{k_i}(t)) \right| \geq 1/2.
\]

{\em Step 2.} Taking a set $E$ from Step 1  we conclude that
\[
\left\| \sum_{i=1}^p  \alpha_i |g_{n_i}-g_{k_i}|  \right\| \geq \left\| \sum_{i\in E } \alpha_i (g_{n_i}-g_{k_i}) \right\| \geq\theta\cdot  \left\| \sum_{i\in E } \alpha_i (e_{n_i}-e_{k_i}) \right\|\geq \theta/2.
\]
This verifies the claim and the proof is complete.
\end{proof}

\begin{example}\label{exa:Ascoli-1}
Suppose that $E$ is a Banach space containing an isomorphic copy of $\ell_1$. Then there is a function $\Phi:B_w\to\IR$ which is sequentially continuous but not continuous.
\end{example}

\begin{proof}
Let $K$ denote the dual unit ball $B_{E^\ast}$ equipped with the $weak^\ast$ topology. Write $Ix$ for the function on $K$ given by $Ix(x^\ast)=x^\ast(x)$ for $x^\ast\in K$.
Then $I:E\to C(K)$ is an isometric embedding.

Since $E$ contains a copy of $\ell_1$, there is a normalized sequence $(x_n)$ in $E$ which is a $\theta$-$\ell_1$-sequence for some $\theta>0$.
Then the functions $g_n=Ix_n$ form a $\theta$-$\ell_1$-sequence in $C(K)$. By Proposition \ref{2:1}  there is a probability measure $\mu$ on $K$  such that
$ \int_K |g_n-g_k| \dx\mu  \ge \theta/2$  whenever  $n\neq k$.

Define a function $\Phi$ on $E$ by $\Phi(x)=\int_K|Ix|\dx\mu.$ If $y_j\to y$ weakly in $E$ then $Iy_j\to I y$ weakly in $C(K)$, i.e.\ $(Iy_j)_j$ is a uniformly bounded sequence converging pointwise to $Iy$. Consequently, $\Phi(y_j)\to \Phi(y)$ by the Lebesgue dominated convergence theorem.
 Thus $\Phi$ is sequentially continuous.

We now check that $\Phi$ is not weakly continuous at $0$ on $B_w$. Consider a basic weak neighbourhood of $0\in B_w$ of the form
\[
V=\{x\in B_w: |x_j^\ast(x)| <\eps \mbox{ for } j=1,\ldots, r\}.
\]
Then there is an infinite set $N\sub\en$ such that $\big(x_j^\ast (x_n)\big)_{n\in N} $ is a converging sequence for every $j\le r$. Hence there are $n\neq k$ such that
$|x_j^*(x_n-x_k)|<\eps$ for every $j\le r$, which means that  $(x_n-x_k)/2 \in V$. On the other hand, $\Phi\big( (x_n-x_k)/2\big)\geq \theta/4$ which demonstrates that $\Phi$ is
not continuous at $0$.
\end{proof}

\section{Proof of Theorem \ref{t:Ascoli-E-weak} and final questions}

In order to prove Theorem \ref{t:Ascoli-E-weak} we need the following two results also of independent interest.
\begin{proposition} [\cite{GKKP}] \label{fre1}
Let $E$ be a metrizable lcs. Then  every bounded subset of $E$ is Fr\'echet--Urysohn in the weak topology of $E$ if and only if every bounded sequence in $E$ has a Cauchy subsequence in the weak topology of $E$.
\end{proposition}
\begin{proposition}[\cite{ruess}]\label{ruess}
Let $E$ be a complete lcs such that every bounded set in $E$ is metrizable. Then  $E$ does not contain a copy of $\ell_{1}$ if and only if every bounded sequence in $E$ has a Cauchy subsequence in the weak topology of $E$.
\end{proposition}

\begin{proof} [Proof of Theorem \ref{t:Ascoli-E-weak}]
(i)$\Rightarrow$(ii): By  Proposition \ref{fre1} and Proposition \ref{ruess} every bounded set $A$ in $E$ is even Fr\'echet-Urysohn in the weak topology of $E$. The converse implication (ii)$\Rightarrow$(i) follows from Theorem  \ref{t:Ascoli-Main}.
\end{proof}

We complete the paper with a few open questions. By Proposition \ref{p:Ascoli-L1-weak}, there is a countable (hence Lindel\"{o}f) non-Ascoli space $A$. So $A$ is homeomorphic to a closed subspace of  some $\IR^\kappa$. As $\IR^\kappa$ is a $k_\IR$-space, we see that  a $k_\IR$-space  may contain a countable closed non-Ascoli subspace.
So the $k_\IR$-space property and the Ascoli property are not preserved in general by closed subspaces.
\begin{question} \label{q:Ascoli-closed-hereditary}
Let $X$ be an Ascoli space such that every closed subspace of $X$ is Ascoli. Is $X$ a $k$-space?
\end{question}

Arhangelskii \cite[3.12.15]{Eng} proved that a topological space $X$ is a hereditarily $k$-space if and only if $X$ is Fr\'{e}chet--Urysohn.
\begin{question} \label{q:Ascoli-hereditary}
Let $X$ be a hereditarily Ascoli space. Is $X$ Fr\'{e}chet--Urysohn?
\end{question}

Let $E=C_p(\omega_1)=\IR^{\omega_1}$. Then the lcs $E$ is a $k_\IR$-space by \cite[Theorem 5.6]{Nob} and is not a $k$-space by \cite[Problem 7.J(b)]{Kelley}. So the $k_\IR$-space property and the Ascoli property are not equivalent to the $k$-space property for $C_p$-spaces, see  the Pytkeev and Gerlits--Nagy Theorem \cite[II.3.7]{Arhangel}.
\begin{question} \label{q:Ascoli-Cp}
For which Tychonoff spaces $X$ the space $C_p(X)$ is Ascoli (or a $k_\IR$-space)?
\end{question}

It is well-known (see \cite[III.1.2]{Arhangel}) that, for a compact space $K$, the space $C_p(K)$ is a $k$-space if and only if $K$ is scattered. Below we generalize this result.
\begin{proposition}\label{p:Ascoli-Cp-scattered}
Let $K$ be a compact space. Then $C_{p}(K)$ is a $k_\IR$-space if and only if $K$ is scattered.
\end{proposition}
\begin{proof}
If $K$ is scattered, then $C_{p}(K)$ is Fr\'echet--Urysohn, and we are done, see \cite[Theorem III.1.2]{Arhangel}.
Now  assume that $K$ is not scattered. Then there is a continuous map $p$ from $K$ onto $[0,1]$ by \cite[8.5.4]{sema}.
Let $\lambda$ be the Lebesgue measure on $[0,1]$. Take a measure $\mu$ on $K$ such that $p[\mu]=\lambda$ (see Lemma \ref{image_measure}). Note that the measure $\mu$ vanishes on points.
If we define
\[
\Psi(g)=\int_X\frac{|g|}{|g|+1}\dx\mu,
\]
then $\Psi$ is easily seen to be sequentially continuous on $C_p(K)$ by the Lebesgue theorem. This implies that $\Psi$ is continuous on every compact subset $\KK$ of $C_p(X)$ (recall that $\KK$ is Fr\'{e}chet--Urysohn, see \cite[Theorem III.3.6]{Arhangel}). On the other hand, it is easy to construct a family $\GG$ of functions $g:K\to [0,1]$ such that $\int_K g\dx\mu\geq 1/2$ and the zero function lies in the pointwise closure of $\GG$, see \cite[Theorem II.3.5]{Arhangel}). This means
that $\Psi$ is not continuous on $C_p(K)$.
\end{proof}

\begin{remark}{\em
Let $\kappa$ be a cardinal number endowed with the discrete topology. Then $C_p(\kappa)=\IR^\kappa$ is a $k_\IR$-space by \cite{Nob}. Recall also that in a model of set theory without weakly inaccessible cardinals,  any sequentially continuous function on  $\IR^\kappa$ is in fact continuous,
 see \cite{Pl93} for further references. }
\end{remark}

Theorem \ref{t:Ascoli-Ck-Metriz} and Proposition \ref{p:Ascoli-Cp-scattered} motivate the following problem.
\begin{question}
Does there exist $X$ such that $\CC(X)$ or $C_p(X)$ is Ascoli but is not a $k_\IR$-space?
\end{question}

For a Tychonoff space $X$ denote by $L(X)$ (respectively, $F(X)$ and $A(X)$) the free locally convex space (the free or the free abelian topological group) over $X$.
\begin{question} \label{q:Ascoli-Free}
Let $L(X)$ ($F(X)$ or $A(X)$) be an Ascoli space. Is $X$ Ascoli?
\end{question}
\begin{question} \label{q:Ascoli-Free-F}
For which metrizable spaces $X$, the groups $F(X)$ and $A(X)$ are Ascoli?
\end{question}

In \cite{Gabr} the first named author proved that the free lcs $L(X)$ over a Tychonoff space $X$ is a $k$-space if and only if $X$ is a discrete countable space.
\begin{question}  \label{q:Ascoli-Free-Discr}
Let $L(X)$ be an Ascoli space. Is $X$ a discrete countable space?
\end{question}
We do not know  the answer even if ``Ascoli'' is replaced by a stronger assumption ``$L(X)$ is a $k_\IR$-space'' (see \cite[Question 3.6]{Gabr}).

\vspace{3mm}
{\bf Acknowledgments}.
The authors are deeply indebted to Professor R.~Pol who sent  to T.~Banakh and the first named author a solution of Problem 6.8 in \cite{BG} (see   Corollary \ref{c:Ascoli-Ck-Pol}). In \cite{Pol-2015},   R.~Pol  noticed that it can be shown that the space $\CC(M)$, where $M$ is the countable metric fan, contains a closed countable non-Ascoli subspace using ideas from \cite{Pol-1974}. Using this fact and stratifiability of metric spaces, for a separable metric space $X$,  R.~Pol  proved that the space $\CC(X)$ is  Ascoli if and only if  $X$ is locally compact. We provide another proof of a more general result by modifying the proof of the assertion in Section 5 of \cite{Pol-1974} (see Proposition \ref{p:Ascoli-Pol}).

\bibliographystyle{amsplain}

\end{document}